\documentclass[11pt]{article}
\usepackage{amssymb,amsmath,amsthm}
\usepackage[usenames]{color}

\DeclareMathOperator{\reg}{reg}
\DeclareMathOperator{\lk}{lk}

\DeclareMathOperator{\rank}{rank}
\DeclareMathOperator{\st}{st}

\DeclareMathOperator{\supp}{supp}
\DeclareMathOperator{\lcm}{lcm}
\theoremstyle{plain}
\newtheorem{theorem}{Theorem}[section]
\newtheorem{proposition}[theorem]{Proposition}
\newtheorem{lemma}[theorem]{Lemma}
\newtheorem*{lemma*}{Auxiliary Lemma}

\newtheorem{corollary}[theorem]{Corollary}

\newtheorem{problem}[theorem]{Problem}

\theoremstyle{definition}

\newtheorem*{acknowledgement}{Acknowledgement}

\title{Regularity via topology of the lcm-lattice for $C_4$-free graphs}
\author{Eran Nevo\footnote{Department of Mathematics, Cornell
University, Ithaca USA, E-mail address:
eranevo@math.cornell.edu. Research partially supported by an NSF Award DMS-0757828.}}
\begin{document}
\maketitle
\begin{abstract}
We study the topology of the lcm-lattice of edge ideals and derive upper bounds on the Castelnuovo-Mumford regularity of the ideals. In this context it is natural to restrict to the family of graphs with no induced $4$-cycle in their complement. Using the above method we obtain sharp upper bounds on the regularity when the complement is a chordal graph, or a cycle, or when the primal graph is claw free with no induced $4$-cycle in its complement. For the later family we show that the second power of the edge ideal has a linear resolution.
\end{abstract}
\section{Preliminaries}\label{sec:pre}
Fix a field $k$ and let $I$ be a monomial ideal in the polynomial ring $S=k[x_1,...,x_n]$ with a minimal set of generators $G(I)=\{m_1,...,m_d\}$. Let $L(I)$ be the lcm-lattice of $I$, i.e. the poset whose elements are labeled by the least common multiples of subsets of monomials in $G(I)$ ordered by divisibility. Indeed $L(I)$ is a lattice, its minimum is $1$ (corresponding to the emptyset) and its maximum is $m_I=\lcm(m:\ m\in G(I))$.

The minimal free resolution of $I$ over $S$ is $\mathbb{N}^n$-graded. Denote the corresponding multi-graded Betti numbers by $\beta_{i,m}$
for a monomial $m$ and $i\geq 0$ an integer.
The main result in \cite{Gasharov-Peeva-Welker} shows how to compute the Betti numbers from the reduced homology of the order complex of open intervals in $L(I)$.

\begin{theorem}\cite[Theorem 2.1]{Gasharov-Peeva-Welker}\label{thm:GPW}
Let $\Delta((1,m))$ denote the order complex of the open interval $(1,m)$ in $L(I)$.
For every $i\geq 0$ and $m\in L(I)$ one has
$$\beta_{i,m}=\dim_k \tilde{H}_{i-1}(\Delta((1,m));k).$$
If $m\notin L(I)$ then $\beta_{i,m}=0$ for every $i$.
\end{theorem}
(In \cite{Gasharov-Peeva-Welker} $S/I$, rather than $I$, was resolved, hence the shift in the index.)
It follows that the Castelnuovo-Mumford regularity of $I$ is
\begin{equation}\label{eq:reg}
\reg(I)=\sup_{i\geq 0}(\max \{j: \exists m\in L(I),\ \deg(m)=i+j,\
\tilde{H}_{i-1}(\Delta((1,m));k)\neq 0 \}).
\end{equation}

Further work on $L(I)$ appeared in \cite{Phan}.
For unexplained terminology on posets, simplicial complexes and topology we refer to Bj\"{o}rner \cite{Bjorner:TopologicalMethods}.

For a graph $G=(V,E)$ let $I(G)$ be its \emph{edge ideal}, namely $I(G)=(x_ix_j:\ \{i,j\}\in E(G))$. This is the case where $G(I)$ consist of squarefree monomials of degree $2$. Denote $m_G=m_{I(G)}$ in this case.
In this paper we consider edge ideals.  These got much attention in recent years, from both algebraists and combinatorialists. E.g., in the recent papers \cite{Engstrom-Dochtermann, Ramos-Gimenez, Woodroofe}
algebraic properties of certain edge ideals are derived from the topology of the clique complex of the complementary graph.
We study the topology of the lcm-lattice of the edge ideal (of the original graph) and its powers, which in turn implies upper bounds on their regularity.

Let $G^c$ be the complement of $G$, namely $G^c=(V, \binom{V}{2}-E)$. When considering $L(I(G))$ it is natural to assume that $G^c$ has no induced $4$-cycles, as is explained in Section \ref{sec:WhyNotC4}, so we restrict our attention to this class of graphs and some subclasses of it.

In Section \ref{sec:chordal} we consider chordal graphs and in Section \ref{sec:cycles} we consider cycles. From our results on the lcm-lattice of their complement we derive a new proof of Fr\"{o}berg's theorem, that $I(G)$ has a linear resolution iff $G^c$ is chordal. Moreover, the main result in \cite{Ramos-Gimenez} also easily follows.
Further, the relation between the homology of the lcm-lattice and the homology of the clique complex of the complementary graph is explained.

It was suggested by Francisco, H\`{a} and Van Tuyl \cite{FHvT} that if $G^c$ has no induced $4$-cycles then for any $k\geq 2$, $I(G)^k$ has a linear resolution.
While this is not true (see \cite{Nevo-Peeva} for examples), it may be true for the subfamily where in addition $G$ is \emph{claw free}, i.e. has no induced bipartite subgraph with one vertex on one side and $3$ vertices on the other.
Note that this family contains all graphs $G$ such that $G^c$ has no induced $3-$ nor $4-$ cycles.
\begin{theorem}\label{thm:lin.res.clow.square}
Let $G$ be claw free such that $G^c$ has no induced $4$-cycle. Then:

(1) $I(G)^2$ has a linear resolution.

(2) If $G^c$ is not chordal then $\reg(I(G))=3$.
\end{theorem}
This seems to be the first infinite family of graphs such that while $I(G)$ does not have a linear resolution, a higher power of it ($I(G)^2$ in this case) does.
Theorem \ref{thm:lin.res.clow.square}
is proved in Section \ref{sec:claw}.

\section{Why not $C_4$?}\label{sec:WhyNotC4}
Let $C_4$ denote 4-cycle.
Recall that a poset $P$ is \emph{pure} if all its maximal chains have the same finite length.
It is shown in \cite[Theorem 2.2]{Nevo-Peeva} that
\begin{proposition}\label{prop:gradedLCM}
If $G^c$ has no induced $C_4$ then for any $k\geq 1$ the lcm-lattice $L(I(G)^k)$ is \emph{pure}, and except for the minimum, the rank function is given by $\rank(m)=\deg(m)-2k+1$.
\end{proposition}
This makes tools from graded poset topology applicable.
In this situation any interval $[x,y]$ in $L(I(G)^k)$ where $x\neq 1$ is a semimodular lattice, and hence shellable \cite[Theorem 3.1]{Bjorner-shellableCMposets}, a fact which we will use in the sequel to derive information on the regularity of $I(G)^k$.

If $G^c$ has an induced $C_4$, equivalently if $G$ has two disjoint edges as an induced subgraph, then $L(I(G)^k)$ is not graded by degree of monomials (up to a shift). Moreover,
\begin{lemma}\label{lem:NoGrading}
If $G$ is a connected graph and $G^c$ has an induced $C_4$ then $L(I(G))$ is not pure.
\end{lemma}
\begin{proof}
As $G$ is connected there is a maximal chain in $[1,m_G]$ of length $|V(G)|$: look on a sequence of edges which form a spanning tree in $G$ and such that every initial segment forms a connected graph. The joins corresponding to initial segments form a maximal chain of length $|V(G)|$.

As $G$ has induced two disjoint edges $\{a,b\},\{c,d\}$ there is a maximal chain in $[1,m_G]$ of length smaller than $|V(G)|$: look on a maximal chain $(1, x_ax_b ,x_ax_b x_cx_d,...)$.
\end{proof}
For the path of length $n\geq 5$, $P_n$,  the lcm-lattice of $P_n$ is not pure then.
It can be shown that e.g. the conclusion of Theorem \ref{thm:regClawFree} on regularity fails for these graphs (which are claw free but contain an induced $C_4$ in the complement). Actually $\reg(I(P_n))\to \infty$ as $n\to \infty$.   See \cite{Kummini} for a detailed analysis.

\section{Chordal graphs}\label{sec:chordal}
A graph is \emph{chordal} if it has no induced cycles of length $>3$.
In particular, chordal graphs have no induced $C_4$.
Dirac characterization of chordal graphs \cite{DiracOrder} implies that if $G^c$ is chordal then the vertices of $G$ can be totally ordered such that if $i,j,k\in V(G)$, $k>i,j$ and $\{i,j\}\in E(G)$ then either $\{i,k\}\in E(G)$
or $\{j,k\}\in E(G)$. Such an order is called a \emph{Dirac order} on $V(G)$.

A pure simplicial complex $\Delta$ is \emph{constructible} if it is a simplex or empty, or inductively, if $\Delta=\Delta_1\cup \Delta_2$ such that $\Delta_1$, $\Delta_2$ and
$\Delta_1\cap \Delta_2$ are constructible and $\dim(\Delta_1)=\dim(\Delta_2)=\dim(\Delta_1\cap \Delta_2) +1$. If $\Delta$ is constructible of dimension $d$ then
$\Delta$ is $(d-1)$-connected; in particular, a nonzero reduced homology
$\tilde{H}_i(\Delta)$ may appear only in dimension $i=d$.

\begin{theorem}\label{thm:Chordal}
If $G^c$ is chordal then $\Delta((1,m_G))$ is constructible.
\end{theorem}
\begin{proof}
If $|V(G)|\leq 3$ or $E(G)=\emptyset$ then the assertion is trivial. For $|V(G)|> 3$ and $E(G)\neq \emptyset$, let $v_1<v_2<...<v_t$ be a Dirac order on $V(G)$. Note that the induced graph on a subset of the vertices of a chordal graph is chordal. By induction the assertion holds for the induced subgraphs $G_l=G[v_1,...,v_l]$ for $2\leq l \leq t-1$.
Let $\Delta_l=\Delta_l(G)$ be the subcomplex of $\Delta((1,m_G))$ spanned by the maximal chains in $(1,m_G)$ whose bottom is an edge contained in $\{v_1,...,v_l\}$. Then
$\Delta((1,m_G))=\Delta_t$. Let $l_0$ be the minimal $i$ such that $\Delta_i\neq \emptyset$.

We now show that  for any $l\geq l_0$ $\Delta_l$ is constructible.
For $l=l_0$ $\Delta_{l_0}=\Delta((x_{v_{l_0}},m_G))$
where $[x_{v_l},m_G]$ is the restriction of $[1,m_G]$ to monomials divisible by $x_{v_l}$, and adding $x_{v_l}$ as a minimum.
Note that $[x_{v_{l_0}},m_G]$ is a semimodular lattice, hence by \cite{Bjorner-shellableCMposets} $\Delta((x_{v_{l_0}},m_G))$ is shellable and in particular constructible. Note that  $\dim(\Delta_{l_0})=\deg(m_G)-3$ by Proposition \ref{prop:gradedLCM}.
For $l>l_0$,
$$\Delta_l=\Delta_{l-1}\cup \Delta((x_{v_l},m_G)).$$
Again, as $[x_{v_l},m_G]$ is a semimodular then
$\Delta((x_{v_l},m_G))$ is constructible. Further,
$\Delta_{l-1}$ is constructible by the induction hypothesis and  $\dim(\Delta_{l-1})=\dim(\Delta((x_{v_l},m_G)))=\deg(m_G)-3$.
The intersection $\Delta_{l-1}\cap \Delta((x_{v_l},m_G))$ is pure of dimension
$\deg(m_G)-4$, and its (nonempty collection of) facets are the maximal chains in $(1,m_G)$ with bottom $x_{v_l}x_{v_i}x_{v_j}$ where $\{v_i,v_j\}\in G$ and $i,j<l$. This follows from the definition of Dirac order. Moreover, $\Delta_{l-1}\cap \Delta((x_{v_l},m_G))$ is combinatorially isomorphic to $\Delta_{l-1}(G[V-\{v_{l}\}])$ which is constructible by the induction hypothesis.
We conclude that $\Delta_t$ is constructible.
\end{proof}

\begin{corollary}\label{cor:ChordalCM}
If $G^c$ is chordal then $\Delta((1,m_G))$ is Cohen-Macaulay.
\end{corollary}
\begin{proof}
Using Reisner theorem, e.g. \cite[Corollary 4.2]{StanleyGreenBook}, we need to show
that for any $F\in \Delta=\Delta((1,m_G))$ and  $i<\deg(m_G)-3-|F|$, the link $\lk_{\Delta}F$ satisfies $\tilde{H}_i(\lk_{\Delta}F)=0$. For $F=\emptyset$ this holds by Theorem \ref{thm:Chordal}.
For $F=\{a_1<...<a_f\}$, $\lk_{\Delta}F=\Delta((1,a_1))*\Delta((a_1,a_2))*...*\Delta((a_{f-1},a_f))*\Delta((a_f,m_G))$ where $*$ denotes join.
By Theorem \ref{thm:Chordal}  $\Delta((1,a_1))$ is constructible and by semimodularity $\Delta((a_i,a_{i+1}))$ and $\Delta((a_f,m_G))$ are shellable, hence in each of these pure complexes only the top dimensional homology group may not vanish. By K\"{u}nneth formula only the top dimensional homology group of their join may not vanish.
\end{proof}

\begin{corollary}\label{cor:ChordalLinRes}\cite{Froberg}
If $G^c$ is chordal then $I(G)$ has a linear resolution.
\end{corollary}
\begin{proof}
By Theorem \ref{thm:Chordal} and equation (\ref{eq:reg}), $\reg(I(G))=2$ hence $I(G)$ has a linear resolution.
\end{proof}
The converse of Corollary \ref{cor:ChordalLinRes}, also proved by Fr\"{o}berg, will follow from Theorem \ref{thm:C_n} in the next section.

\section{Induced cycles}\label{sec:cycles}
If $H$ is an \emph{induced} subgraph of $G$, then $[1,m_H]$ is an interval in
$[1,m_G]$ and hence $\reg(I(H))\leq \reg(I(G))$.
If $G^c$ is not chordal then it contains an induced cycle $H^c=C_n$, of length $n>3$.
If $H^c=C_4$ then $\Delta((1,m_H))$ is the zero dimensional sphere while $\deg(m_H)=4$ hence $3\leq \reg(I(G))$, thus $I(G)$ does not have a linear resolution.
What happens if $n\geq 5$?:

\begin{theorem}\label{thm:C_n}
Let $n\geq 3$ and $G^c=C_n$. Then $H_{\ast}(\Delta((1,m_G)))\cong H_{\ast}(\mathbb{S}^{n-4})$, where $\mathbb{S}^d$ is the
$d$-dimensional sphere.
\end{theorem}
\begin{proof}
For $n=3$ the assertion is trivial. For $n\geq 4$ let $\Delta$ be the barycentric subdivision of the boundary of the simplex on $n$ vertices. Thus, the vertices of $\Delta$ are labeled by the proper nonempty subsets of $[n]$ and its faces correspond to chains of subsets ordered by inclusion.
Let $\Gamma$ be the induced subcomplex of $\Delta$ with vertex set $V$ consisting of all singletons, all consecutive pairs $\{i,i+1\}$ and all consecutive triples $\{i-1,i,i+1\}(\mod n)$ in $[n]$.

One easily checks that $\Gamma$ deformation retracts on $C_n$ (retract the triangles with vertex $\{i-1,i,i+1\}$ on the length $2$ path $(i-1,i,i+1)$).
As $\Gamma$ is induced, $\Delta - \Gamma$ deformation retracts onto the induced subcomplex on the complementary set of vertices $\Delta[V(\Delta)-V(\Gamma)]$.
As $\Delta$ is a $(n-2)$-sphere, it follows from Alexander duality
\cite[Chapter 8, \S 71]{Munkres:AlgebraicTopology-1984} that for every $i$
$$\tilde{H}_i(\Delta[V(\Delta)-V(\Gamma)])\cong
\tilde{H}_i(\Delta-\Gamma)\cong
\tilde{H}^{n-3-i}(\Gamma)\cong \tilde{H}^{n-3-i}(\mathbb{S}^1)
.$$
By the obvious bijection between subsets of $[n]$ and square free monomials with variables in $\{x_1,...,x_n\}$, we get a combinatorial isomorphism $\Delta((1,m_G))\cong \Delta[V(\Delta)-V(\Gamma)]$, and hence
$H_{i}(\Delta((1,m_G)))\cong H_{i}(\mathbb{S}^{n-4})$
for all $i$.
\end{proof}

\begin{corollary}\label{cor:C_nLinRes}\cite{Froberg}
If $G^c$ is not chordal then $I(G)$ does not have a linear resolution.
\end{corollary}
\begin{proof}
By Theorem \ref{thm:C_n} and (\ref{eq:reg}), there is some $n\geq 4$ such that $\reg(I(G))\geq \reg(C_n)>2$ hence $I(G)$ does not have a linear resolution.
\end{proof}
In the recent papers \cite{Woodroofe, Engstrom-Dochtermann, Ramos-Gimenez}
Hochster formula (Theorem \ref{thm:Hochster}), applied to the clique complex of $G^c$, was used to derive Fr\"{o}berg's theorem, i.e. Corollaries \ref{cor:ChordalLinRes} and \ref{cor:C_nLinRes}. The main result in the later reference, namely \cite[Theorem 1.1]{Ramos-Gimenez}, easily follows from Theorems \ref{thm:Chordal} and \ref{thm:C_n}.

We show now that as far as homology is concerned, Hochster formula and the lcm method are equivalent. More precisely:
\begin{proposition}\label{prop:lcmHochster}
Let $E(G)\neq \emptyset$, $|V(G)|=n$, $k$ a field and denote by $\Delta(G^c)$ the clique complex on $G^c$.
Then for any $i$ the reduced homology groups satisfy
$$\tilde{H}_{i}(\Delta(G^c);k)\cong \tilde{H}_{n-3-i}(\Delta((1,m_G));k).$$
\end{proposition}
\begin{proof}
Let $C$ be the set of minimal non faces of $\Delta(G^c)$. Then $C=E(G)$. Let $\Gamma$ be the simplicial complex on the vertex set $C$ with faces $F$ such that $\cup_{u\in F}u\neq V(G)$. As $V(G^c)\notin \Delta(G^c)$, by \cite[Theorem 2]{Bjorner-NoteAlexanderDuality}
\begin{equation}\label{eq:HomolBj}
\tilde{H}_{i}(\Delta(G^c);\mathbb{Z})\cong \tilde{H}^{n-3-i}(\Gamma;\mathbb{Z})
\end{equation}
for all $i$, where $\tilde{H}^{j}$ denotes the $j$-th cohomology group.

To show that $\Gamma$ is homotopy equivalent to $\Delta((1,m_G))$ consider $\Gamma':=\Gamma - \{\emptyset\}$ as a poset where faces are ordered by inclusion, and the poset map
$$\pi: \Gamma' \longrightarrow (1,m_G),\ \pi(F)=\prod\{x_i:\ i\in \cup_{u\in F}u\}.$$
Note that $\pi$ is onto. For $W\subsetneq V(G)$ such that $x_W:=\prod_{i\in W}x_i \in (1,m_G)$ the fiber $\pi^{-1}(\{y: y\leq x_W\})$ has a unique maximal element  $\{c\in C: \ c\subseteq W\}$, hence its order complex is contractible. By Quillen's fiber theorem \cite[Proposition 1.6]{Quillen} the barycentric subdivision of $\Gamma$ is homotopic to $\Delta((1,m_G))$, and hence so is $\Gamma$. Working over a field, the isomorphism between homology and cohomology together with (\ref{eq:HomolBj}) imply the result.
\end{proof}
\begin{problem}\label{prob:C_n}
Let $G^c=C_n$ for $n\geq 4$. Is $\Delta((1,m_G))$ homotopic to $\mathbb{S}^{n-4}$?
\end{problem}

\section{Claw free graphs}\label{sec:claw}
A graph $G$ is \emph{claw free} if it contains no $4$ vertices on which the induced graph is a \emph{star}, i.e. a connected graph where all vertices but one have exactly one neighbor, which is common to all of them. Claw free graphs are of great interest in combinatorics. The connectivity of the independence complex of claw free graphs was studied in \cite{Engstrom-ClawFree}; in particular it follows that a nonzero homology in the independence complex of $G$, which is the clique complex of $G^c$, can occur in arbitrarily high dimension. Using Hochster formula it means that $\sup\{\reg(I(G)): G\  \rm{is\ claw\ free}\}=\infty$.

If we restrict to claw free graphs with no induced $C_4$ in their complement, denote this family by $\mathcal{CF}$, the situation is drastically different, as Theorem \ref{thm:regClawFree} below shows.

\begin{theorem}\label{thm:regClawFree}
If $G\in \mathcal{CF}$ then $\reg(I(G))\leq 3$.
\end{theorem}
As we have seen, both values of the regularity permitted by this theorem are possible: if $G^c$ is a tree then $\reg(I(G))=2$ and if $G^c=C_n$ for $n\geq 5$ then $\reg(I(G))=3$.
\begin{proof}[Proof of Theorem \ref{thm:regClawFree}]
We shall make use of Hochster formula:
\begin{theorem}(Hochster formula)\label{thm:Hochster}\cite[Corollary 4.9]{StanleyGreenBook}
For a simplicial complex $\Delta$ on vertex set $V$, the Betti numbers of its Stanley-Reisner ideal $I_{\Delta}$ over a field $k$ satisfy for every $i>0$
$$\beta_{i,i+j}(I_{\Delta})=\sum_{W\subseteq V,\ |W|=i+j}\dim_k(\tilde{H}_{j-1}(\Delta[W];k)).$$
\end{theorem}
Thus, to prove the theorem we need to show that for every $l>1$ and every $W\subseteq V(G)$, $\tilde{H}_{l}(\Delta[W];k))=0$ where $\Delta[W]$ is the clique complex on the induced graph $G^c[W]$.

If $\Delta[W]$ has no $2$-dimensional faces, this is obvious. Assume that $F=\{a,b,c\}$ is a $2$-face of $\Delta[W]$. Decompose the geometric realization $|\Delta[W]|$ as a union of two open spaces $\Delta[W]=(|\Delta[W]|-|F|)\cup (\cup_{v\in F}\st(v))$ where $\st(v)$ is the open star of $v$ in $\Delta[W]$. Let $L=(|\Delta[W]|-|F|)\cap (\cup_{v\in F}\st(v))$.
Then $|\Delta[W]|-|F|$ retracts on $\Delta[W-F]$ and by induction on the number of vertices all of its homology groups in dimension $>1$ vanish.
Note that $\cup_{v\in F}\st(v)$ is contractible.
The intersection $L$ is homotopic to $\cup_{v\in F}\lk(v) - \partial F$, where $\partial F$ is the boundary of the simplex with vertex set $F$, and in turn is homotopic to the complex $M$ generated by the faces $T\subseteq (W-F)$ such that one of the sets $T\cup \{v\}$ where $v\in F$ is in $\Delta[W]$.

We now show that if $|T|\geq 3$ then $T\in M$ iff $T\in \Delta[W-F]$, and hence, again by induction, the homology groups of $M$ and hence of $L$ in dimension $>1$ vanish.
Clearly $M\subseteq \Delta[W-F]$. Let $T\in \Delta[W-F]$, $|T|\geq 3$ and assume by contradiction that $T\notin M$, i.e. for any $v\in F$ the number of its neighbors in $G^c$ among $T$ is smaller than $|T|$; w.l.o.g. let $a$ maximize this number among the elements of $F$, and denote this number by $t$ and the neighbors of $v$ in $G^c$ among $T$ by $T(v)$.
Let $u\in T-T(a)$. By claw freeness $u$ has a neighbor in $F$, and w.l.o.g. let $b$ be such neighbor. We will show now that $T(a)\uplus\{u\}\subseteq T(b)$, a contradiction to the choice of $a$: for each $w\in T(a)$, look at the $4$-cycle $(a,b,u,w)$ in $G^c$ and conclude that $\{w,b\}\in G^c$, hence $w\in T(b)$.

By Mayer-Viatoris long exact sequence over $\mathbb{Z}$ we get for $i>1$
$$0=H_i(|\Delta[W]|-|F|)\oplus H_i(\bigcup_{v\in F}\st(v))\rightarrow H_i(\Delta[W])\rightarrow H_{i-1}(L) \stackrel{j_*}{\rightarrow} H_{i-1}(|\Delta[W]|-|F|)
,$$
thus we will be done if we show that the connecting homomorphism $j_*$ is injective.
This will follow from showing that the diagram
$$\begin{array}{ccccccccc}\label{commDiagram}
   H_l(L) & \stackrel{j_*}{\rightarrow} & H_l(|\Delta[W]|-|F|)  \\
   \cong \downarrow & &\ \ \downarrow \cong \\
   H_l(M) & \stackrel{i_*}{\rightarrow} & H_l(\Delta[W-F])\\
   \end{array}$$
commutes for any $l$, where $i_*$ is induced by inclusion. Indeed, we already showed that $i_*$ is injective for $l\geq 1$.
Commutativity follows from taking a retract $|\Delta[W]|-|F| \to \Delta[W-F]$ whose restriction to $L$ is a retract onto $M$; this is easy to do, we omit the details.
\end{proof}

Denote by $L_{i}$ the restriction of the lcm lattice $L=L(M)$ to monomials of degree at least $i$ (not to be confused with the rank of them as elements in the poset).
For a simplicial complex $\Gamma$ let $\alpha(\Gamma)$
be the maximal number such that $\tilde{H}_{\dim(\Gamma)-\alpha(\Gamma)}(\Gamma)\neq 0$, and set
$\alpha(\Gamma)=0$ if $\Gamma$ is acyclic.
For a monomial $m$ in $L$ let $\alpha(m)=\alpha((1,m)):=\alpha(\Delta((1,m)))$.
Let $\alpha(M):=\max_{1\neq m\in L(M)}\{ \alpha(m) \}$. If $M$ is generated by monomials of degree $r$ then $\reg(M)=r + \alpha(M)$ (use (1), or see \cite[Proposition 2.3]{Nevo-Peeva}). Also, denote $\supp(m)=\{v: x_v | m\}$.

The following proposition was suggested to me by Irena Peeve, generalizing a result of Phan \cite{Phan} who proved the case where $i=3$ and $M$ has a linear resolution. It will be used in the proof of Theorem
\ref{thm:lin.res.clow.square}(1).

\begin{proposition}\label{prop:genPhan}.
Let $M$ be a monomial ideal minimally generated by monomials of degree $s\ge 2$. Suppose that its lcm-lattice $L(M)$ is graded
and except for the minimum, the rank function is given by $\hbox{rank}(m)=\deg(m)-s+1$ ($m$ is a monomial).  Suppose that there exist monomials of degree $s+1$ in $L(M)$, and let
$Q$ be the monomial ideal generated by all such monomials, that is, $Q$
is generated by the multidegrees of the first minimal syzygies of $M$.
Then for any $m\in Q$, $\alpha((1,m)_{L(Q)})\le \max(0, \alpha((1,m)_{L(M)}-1)$.
In particular,
$$\reg(Q)\le \max(s+1, \reg(M)) .$$
\end{proposition}

\begin{proof}
Fix a monomial $m\in L(M)$. Let $A$ be the set of atoms in $(1,m)_{L(M)}$, $\Delta=\Delta((1,m)_{L(M)})$, $\Gamma=\Delta-A$, and $\Lambda$ the induced subcomplex of $\Delta$ on the complement of $A$, i.e. $\Lambda=\Delta((1,m)_{L(Q)})$.
Then $\Gamma$ deformation retracts on $\Lambda$ and $\dim(\Lambda)=\dim(\Gamma)-1$.

The Mayer-Vietoris sequence gives
$$\tilde{H_i}\bigl(\, \uplus_{a\in A}\lk(a,\Delta)\bigr)\longrightarrow\
\tilde{H_i}(\Gamma)\oplus \tilde{H_i}\bigl(\,\uplus_{a\in A}\st(a,\Delta)\bigr)\longrightarrow
\tilde{H_i}(\Delta)\, .$$

For any $a\in A$,  the link $\lk(a,\Delta)=\Delta((a,m))$ is shellable (by \cite[Theorem 3.1]{Bjorner-shellableCMposets} again). Therefore, we get that $\alpha(\lk(a,\Delta))=0$,
hence $\alpha(\Gamma)\leq \max(1,\alpha(\Delta))$. Now,
the assertion follows as $\dim(\Lambda)=\dim(\Gamma)-1$.
\end{proof}

\begin{corollary}\label{lem:L_i}
If $G^c$ has no induced $C_4$, then for every $m\in L=L(I_G)$ and $i\geq 2$
$$\alpha((1,m)_{L_{i}})\leq \max(0, \alpha((1,m)_{L})-i+2).$$
\end{corollary}
\begin{proof}
Combine Propositions \ref{prop:gradedLCM} and \ref{prop:genPhan}.
\end{proof}

The following theorem is a restatement of Theorem \ref{thm:lin.res.clow.square}(1).

\begin{theorem}\label{thm:eregI^2(ClawFree)}
If $G\in \mathcal{CF}$ then $\reg(I(G)^2)=4$.
\end{theorem}
\begin{proof}
By (1) and Proposition \ref{prop:gradedLCM} we need to show that
$\alpha(m)=0$ holds for any $m\in L^2:=L(I(G)^2)$.

If $|\supp(m)|\leq 3$ then one easily checks that $\alpha(m)=0$ (note that any variable appears in degree at most $2$ in $m$). So assume that $|\supp(m)|\geq 4$. As $G[\supp(m)]$ is claw free, it contains two disjoint edges, and their product divides $m$. Let $m_{sf}$ be the (nonempty) join of squarefree atoms in $(1,m]$.
We distinguish two cases.

Case 1: $m_{sf}=m$.
We need the following lemma.
\begin{lemma}\label{lem:sf-partReduction}
Let $G\in \mathcal{CF}$ and $m\in L^2:=L(I(G)^2)$ be squarefree. Let $L:=L(I(G))$.
Then $(1,m]_{L^2}=(1,m]_{L_{4}}$.
\end{lemma}
\begin{proof}
As in both posets the elements are all the joins of monomial of degree $4$, it is enough to show that a monomial of degree $4$ is in $(1,m]_{L^2}$ iff it is in $(1,m]_{L_{4}}$.
Let $m'$ be a monomial of degree $4$. If $m'\in (1,m]_{L^2}$ clearly $m'\in (1,m]_{L_{4}}$. Conversely, if $m'\in (1,m]_{L_{4}}$ then $G[\supp(m')]$ contains two (not induced!) disjoint edges as $G$ is claw free, and their product shows $m'\in (1,m]_{L^2}$.
\end{proof}
Back to the proof of Case 1, combining Lemma \ref{lem:sf-partReduction}, Corollary \ref{lem:L_i} and Theorem \ref{thm:regClawFree} gives $\alpha((1,m)_{L^2})\leq \max (0,\, \alpha((1,m)_L)-2) = 0$ as desired.

\

Case 2: $m_{sf}\neq m$.
For an induced subposet $L$ of an lcm lattice generated by monomials of degree $4$, denote by $L_{\neg 2}$ the restriction of $L$ to the joins of atoms which are not squares, i.e. not of the form $(ab)^2$. First we show that:
\begin{lemma}\label{lem:NoPureSquare}
For any $m\in L(I(G)^2)$, $\alpha((1,m))\leq \alpha((1,m)_{\neg 2})$.
\end{lemma}
We postpone its proof for later.
To conclude in Case 2, it is enough to show that
$\alpha((1,m)_{\neg 2})=0$.

Let $P_0=(1,m_{sf}]$ and for $i>0$ let $P_i$ be the restriction of  $(1,m)_{\neg 2}$ to $P_0$ union with the elements of degree at least $\deg(m)-i$ in $(1,m)_{\neg 2}$. Then $P_0\subseteq P_1\subseteq...\subseteq P_{\deg(m)-4}=(1,m)_{\neg 2}$. Note that $\Delta(P_0)$ as acyclic as it is a cone.

We will show first that $\Delta(P_i)$ is acyclic for $0\leq i\leq \deg(m)-7$.
Let $i>0$ and $x\in P_i - P_{i-1}$. Then $$\lk(x,\Delta(P_i))=\Delta((x,m)_{\neg 2})*\Delta((1,x_{sf}])$$
where $\Delta((1,x_{sf}])=\emptyset$ if $x_{sf}$ does not exist. However, recall that claw freeness guarantees that $x_{sf}$ exists if $|\supp(x)|\geq 4$ which is the case if $\deg(x)>6$. If $x_{sf}$ exists then  $\lk(x,\Delta(P_i))$ is acyclic.

Let $1\le i\le \deg(m)-7$. Order the vertices in $P_i-P_{i-1}$, say $x_1, x_2,..., x_j$. Let $P_{x_l}$ be the induced poset of $(1,m)$ on $P_{i-1}\cup \{x_1,...,x_{l}\}$ and $\Delta(P_{x_l})$ be its order complex. Define $P_{x_0}:=P_{i-1}$. Let $1\le l \le j$ and by induction we assume that $\Delta(P_{x_{l-1}})$ is acyclic. Consider the Mayer-Vietoris long exact sequence for the union
$\Delta(P_{x_l})= (\Delta(P_{x_l})-\{x_l\}) \cup \st(x_l, \Delta(P_{x_l}))$.
Note that
$\Delta(P_{x_l})-\{x_l\}$ is homotopic to $\Delta(P_{x_{l-1}})$,
$\st(x_l, \Delta(P_{x_l}))$ is acyclic, and their intersection is homotopic to
$\lk(x_l,\Delta(P_i))$ which is a cone. We conclude that
$\Delta(P_{x_l})$ is acyclic too.


Thus, $\Delta(P_{deg(m)-7})$ is acyclic.
For $x_l\in P_{deg(m)-6}-P_{deg(m)-7}$, if $(x_l)_{sf}$ exists then as we showed before, adding it to the poset $P_{x_{l-1}}$ will not effect the homology. If
$(x_l)_{sf}$ does not exist then $\lk(x_l, \Delta(P_{x_{l-1}})) =
\Delta( (x_l,m)_{\neg 2})$ which is shellabe (as $[x_l,m]_{\neg 2}$ is semimodular and see Section 2), hence adding $x_l$ to $P_{x_{l-1}}$ may create nontrivial homology only in dimension $\dim(\Delta(1,m)) -3$. Thus, the Mayer-Vietoris sequence shows that $\Delta(P_{x_l})$ may have nonzero homology only in dimension $\dim(\Delta(1,m))-2 = \deg(m)-7$.
Moreover, it shows that
$\tilde{H}_{\deg(m)-7}(\Delta(P_{\deg(m)-6}))\cong \mathbb{Z}^k$ where $k$ is the number of monomials $x\in P_{\deg(m)-6}\setminus P_{\deg(m)-7}$ such that $\Delta((x,m))$ has nonvanishing top dimensional homology.

Note that for such $x$ $\Delta((x,m))$ is a pseudomanifold (indeed every chain $x<c_1<...<c_{\deg(m)-\deg(x)-2}<m$ is contained in at most two maximal chains in $[x,m]_{\neg 2}$). It follows that for $x$ as above $\Delta((x,m)_{\neg 2})$ is a sphere.
As a representative of the homology induced by $x$ we need to find a cycle in $\Delta(P_{\deg(m)-6})$ (actually we will find a sphere) whose support contains the ball $\Delta([x,m)_{\neg 2})$. For this, we need the following lemma.

\begin{lemma}\label{lem:forPrism}
Let $x\in\{a^2b^2c, a^2b^2c^2\}$ and $y\in (x,m)_{\neg 2}$ where $a,b,c$ are different variables.
Then $y/a \in (x/a,m)_{\neg 2}$.
\end{lemma}
\begin{proof}
Note that $G[\supp(x)]$ is a triangle, hence $x/a\in (1,m)_{\neg 2}$.
If $|\supp(y)|=3$ then $y=a^2b^2c^2$ and the claim is clear. So assume $|\supp(y)|>3$, and as we argued before, claw freeness guarantees the existence of $y_{sf}$. If a variable $v\neq a$ appears in degree $2$ in $y$, then there are two different edges containing the vertex $v$ in $G[\supp(y)]$, and their product, denoted by $e(v)$ is in $(1,m)_{\neg 2}$. Thus, the join of $y_{sf}$ with all the $e(v)$ for $v$ as above equals $y/a$ and is in $(x/a,m)_{\neg 2}$.
\end{proof}

Back to the proof of Theorem \ref{thm:eregI^2(ClawFree)}, we need to consider $x\in P_{deg(m)-6}-P_{deg(m)-7}$ with $\Delta((x,m))$ not acyclic and such that $|\supp(x)|\leq 3$, hence $x=a^2b^2c^2$.
By Lemma \ref{lem:forPrism}, $\{y/a : y\in (x,m)_{\neg 2}\} \subseteq P_{\deg(m)-6}$. The join of these $y/a$ is $m/a \in P_{\deg(m)-6}$.
For each facet $\{c_1<...<c_l\}$ of $\Delta((x,m)_{\neg 2})$, triangulate the prism with top $\{c_1<...<c_l\}$ and bottom $\{c_1/a<...<c_l/a\}$ in the standard way using the facets $\{c_1/a<...<c_i/a<c_i<...<c_l\}$. The union of all these prisms and $\Delta([x,m)_{\neg 2})$ and $\Delta((x/a,m/a]_{\neg 2})$ is a sphere of codimension $2$ in $\Delta(P_{\deg(m)-6})$ representing the nontrivial homology induced by $x$.

On the other hand, the cone over this sphere with apex $x/a$ shows that the map
$$\tilde{H}_{\deg(m)-7}(\Delta(P_{\deg(m)-6})\longrightarrow \tilde{H}_{\deg(m)-7}(\Delta(P_{\deg(m)-5}),$$
induced by the Mayer-Vietoris sequence
for the union
$$\Delta(P_{\deg(m)-5})=(\Delta(P_{\deg(m)-5})- (P_{\deg(m)-5}-P_{\deg(m)-6})) \cup (\bigcup_{x\in P_{\deg(m)-5}-P_{\deg(m)-6}}\st(x, \Delta(P_{\deg(m)-5})),$$
 is zero.
Arguing as before with the Mayer-Vietoris sequence, $\Delta(P_{\deg(m)-5})$
may have nonvanishing homology only in dimension $\deg(m)-6$ (i.e. codimension $1$), and by Lemma \ref{lem:forPrism} applied to $x=a^2b^2c$ and the above argument, this homology maps to zero in $\Delta(P_{\deg(m)-4})$. Thus, applying again the Mayer-Vietoris sequence, $\tilde{H}_i(\Delta(P_{\deg(m)-4}))$ may be nonzero only if $i=\deg(m)-5$ (depending on whether there are atoms $x\in (1,m)_{\neg 2}$ such that $\Delta(x,m)_{\neg 2}$ is a sphere), i.e. $\alpha((1,m)_{\neg 2})=0$.
\end{proof}

\begin{proof}[Proof of Lemma \ref{lem:NoPureSquare}]
If $m$ has degree $4$ there is nothing to prove, as both posets are empty.
Otherwise, $|\supp(m)|>2$ and hence there is an atom below $m$ which is not a square.
Let $m_{ns}$ be the join of all such atoms.
Let $P_0=(1,m)_{\neg 2}$ (it is not empty!) and for $i>0$ let $P_i$ be the restriction of  $(1,m)$ to $P_0$ union with the elements of degree at least $\deg(m)-i$ in $(1,m)$. Then $P_0\subseteq P_1\subseteq...\subseteq P_{\deg(m)-4}=(1,m)$.

Let $i>0$ and $x\in P_i\setminus P_{i-1}$. Then $$\lk(x,\Delta(P_i))=\Delta((x,m))*\Delta((1,x_{ns}]_{\neg 2})$$
where $\Delta((1,x_{ns}]_{\neg 2})=\emptyset$ iff $x_{ns}$ does not exist.
Note that as long as $\deg(x)>4$ then $x_{ns}$ exists, hence $\lk(x,\Delta(P_i))$ is acyclic.

Similar Mayer-Vietoris sequences to the ones using $x_{sf}$ in the proof of Theorem \ref{thm:eregI^2(ClawFree)}(Case 2) show that
$\Delta(P_i)$ is homologic to $\Delta(P_0)$ for $0\leq i\leq \deg(m)-5$.

Let $x\in P_{\deg(m)-4}\setminus P_{\deg(m)-5}$. Then $\lk(x,\Delta(P_{\deg(m)-4}))=\Delta((x,m))$, which is shellable, hence $\alpha((x,m))=0$.

Now add the vertices $\{x_1,...,x_j\}=P_{\deg(m)-4}-P_{\deg(m)-5}$ to $P_{\deg(m)-5}$ one by one, denoting by $P_{x_l}$ the induced poset in $(1,m)$ on $P_{\deg(m)-5}\cup \{x_1,...,x_l\}$, where $P_{x_0}:=P_{\deg(m)-5}$.
The Mayer-Vietoris sequence shows that the homology of $\Delta(P_{x_l})$ may differ from the homology of $\Delta(P_{x_{l-1}})$ only in the top dimension and in codimension $1$, where a difference in codimension $1$ is possible only if the codimension $1$ homology group of $\Delta(P_{x_{l-1}})$ is nonzero. Inductively, this implies that $\alpha(P_{\deg(m)-4})\le \alpha(P_0)$, i.e. that
$\alpha((1,m)) \leq \alpha((1,m)_{\neg 2})$.
\end{proof}
We remark that if $m_{ns}<m$ then the proof above gives that $\alpha((1,m))=0$ as $(1,m)_{\neg 2}$ is a cone (with apex $m_{ns}$).

Theorem \ref{thm:lin.res.clow.square} readily follows from Theorems \ref{thm:regClawFree} and \ref{thm:eregI^2(ClawFree)}.


\begin{acknowledgement}
Many thanks to Irena Peeva for many helpful discussions.
\end{acknowledgement}
\bibliographystyle{amsplain}
\bibliography{biblio,ubt,topology}

\end{document}